\documentclass[10 pt]{amsart}

\usepackage{amsmath} 
\usepackage{amsfonts}
\usepackage{amssymb}
\usepackage{amstext}
\usepackage{amsbsy}
\usepackage{amsthm}
\usepackage{mathtools}
\usepackage{color}
\usepackage{hyperref}
\usepackage{enumitem}

\newtheorem{theorem}{Theorem}[section]
\newtheorem{question}[theorem]{Question}
\newtheorem{conjecture}[theorem]{Conjecture}
\newtheorem{lemma}[theorem]{Lemma}
\newtheorem{prop}[theorem]{Proposition}
\newtheorem{cor}[theorem]{Corollary}
\newtheorem{example}[theorem]{Example}
\newtheorem*{conj1'}{Conjecture 1'}

\theoremstyle{definition}
\newtheorem{definition}[theorem]{Definition}
\newtheorem{notation}[theorem]{Notation}
\newtheorem{defn}[theorem]{Definition}
\newtheorem{remark}[theorem]{Remark}

\newcommand{\Z}{{\mathbb Z}}

\newcommand{\N}{{\mathbb N}}

\newcommand{\CA}{\mathcal{A}}
\newcommand{\CS}{\mathcal{S}}
\newcommand{\CT}{\mathcal{T}}

\newcommand{\bfm}{\mathbf{m}}
\newcommand{\bfn}{\mathbf{n}}
\newcommand{\bn}{\mathbf{n}}
\newcommand{\bfv}{\mathbf{v}}
\newcommand{\bfu}{\mathbf{u}}

\newcommand{\bfe}{\mathbf{e}}

\newcommand{\vre}{\varepsilon}
\newcommand{\diam}{\operatorname{diam}}
\newcommand{\proj}{\operatorname{proj}}

\newcommand{\ext}{\operatorname{ext}}

\newcommand{\ignore}[1] {}

\newcommand{\conv}{{\rm conv}}

\newcommand{\rst}[1]{\ensuremath{{\mathbin\upharpoonright}%
\raise-.5ex\hbox{$#1$}}}

\title{Complexity and directional entropy in two dimensions}

\author{Ryan Broderick}
\address{Northwestern University, Evanston, IL 60208 USA}
\email{ryan@math.northwestern.edu, kra@math.northwestern.edu}
\author{Van Cyr}
\address{Bucknell University, Lewisburg, PA 17837 USA}
\email{van.cyr@bucknell.edu}
\author{Bryna Kra}
\thanks{The  third author was partially supported by NSF grant $1200971$.}

\begin{document}
\begin{abstract} 
We study the directional entropy of the dynamical system associated to a $\Z^2$ 
configuration in a finite alphabet.  We show that under local assumptions on the complexity, either every direction has zero topological entropy or some direction is periodic.  
In particular, we show that all nonexpansive directions in a $\Z^2$ system with the 
same local assumptions have zero directional 
entropy.  
\end{abstract} 

\maketitle

\section{Introduction}

A classic problem in dynamics is to deduce global properties of a system from local assumptions.  A beautiful example 
of such a result is the Morse-Hedlund Theorem~\cite{MH}: a local assumption on the {\em complexity} of a system is equivalent 
to the global property of {\em periodicity} of the system.  Any periodic system trivially has zero topological entropy.  In higher dimensions, this local to global connection is less well understood.  Again, there is a natural local assumption on the system that implies zero topological entropy, but now one can study the finer notion of directional behavior and new subtleties arise: under this assumption some directions may have positive directional entropy, while others do not.  We prove that there are natural local assumptions on the complexity of a $\Z^2$ system under which either every direction has zero topological directional entropy or some direction is periodic.  

To explain the results more precisely, for a finite alphabet $\CA$, we study functions of the form $\eta\colon\Z^2\to\CA$ which we  view as colorings of $\Z^2$.  For $\bfn\in\Z^2$, define 
the translation $T^{\bfn}\colon\Z^2\to\Z^2$ by $T^{\bfn}({\bf x}):={\bf x}+\bfn$ and for fixed $\eta\colon\Z^2\to\CA$, define $T^{\bfn}\eta\colon\Z^2\to\CA$ by $T^{\bfn}\eta({\bf x}):=\eta(T^{\bfn}{\bf x})$.  If $\mathcal{S}\subset\Z^2$, then an $\eta$-coloring of $\mathcal{S}$ is any function of the form $T^{\bfn}\eta\rst{\mathcal{S}}$, where by $\eta\rst{\CS}$ 
we mean the restriction of the coloring $\eta$ of $\Z^2$ to the set $\CS$.  
If $K\subset\mathbb{R}^2$ is compact, we define the complexity $P_\eta(K)$ to be
the number of distinct $\eta$-colorings of $K \cap \Z^2$:
$$P_\eta(K) = \bigl\vert\{T^{\bfn}\eta\rst{K \cap \Z^2} \colon  \bfn \in \Z^2\}\bigr\vert,$$
where $\vert\cdot\vert$ denotes the cardinality.  
This is a generalization to two dimensions of the one dimensional complexity $P_\alpha\colon \N\to\N$ defined 
for $\alpha\colon\Z\to\CA$, where $P_\alpha(n)$ is defined to be the number of distinct words of length $n$ appearing 
in $\alpha$.  

It follows immediately from the definition that $P_\eta$ is monotonic: If $A \subseteq B \subset \mathbb{R}^2$ are 
compact sets,
then $P_\eta(A) \leq P_\eta(B)$. 
When $K$ is the rectangle $[0,n-1] \times [0,k-1]$, we write
$P_\eta(n,k)$ instead of $P_\eta(K)$.

There is a standard dynamical system associated to a configuration such as $\eta\colon\Z^2\to\CA$.  
We endow $\CA$ with the discrete topology and 
$\CA^{\Z^2}$ with the product topology. 
For $\eta\colon\Z^2\to\CA$, we let $X_\eta$ denote the orbit closure 
of $\eta$ under the $\Z^2$ translations $\{T^{\bfn}\colon\bfn\in\Z^2\}$.
Then $X_\eta$, endowed with a distance $\rho$ defined by 
$$\rho(x,y) = 2^{-\min \{\|\bfm\| \colon  x(\bfm) \neq y(\bfm)\}}$$ 
for $x, y\in \CA^{\Z^2}$, 
and with the $\Z^2$ action by translation, is a $\Z^2$ topological dynamical system. 
We refer to an element of the system $X_\eta$ as an {\em $\eta$-coloring} of $\Z^2$.

We say that $\eta\colon\Z^2\to\CA$  is {\em periodic} if there exists $\bfm\neq\bf{0}$ such that $\eta(\bfm+\bfn) = \eta(\bfn)$ for all $\bfn\in\Z^2$ (note that this means that $\eta$ has a direction of periodicity but is not necessarily doubly periodic). 
Again, this is a two dimensional generalization of periodicity for some $\alpha\colon\Z\to\CA.$
It was conjectured by Nivat~\cite{nivat} that for $\eta\colon\Z^2\to\CA$,  if there exist $n, k \in \N$ such that $P_\eta(n,k) \le nk$,
then $\eta$ is periodic.  
This question remains open, but in~\cite{CK} it is observed that a system with
such a complexity bound has zero topological entropy. 

In this work, we study the finer notion of directional entropy, which was 
introduced for cellular automata by Milnor~\cite{milnor} (see Section~\ref{sec:directional} for the definition).
If the directional entropy is finite in all directions, then the system has zero topological entropy, but the converse is false: 
zero topological entropy does not imply anything more than the existence of a single direction with finite directional entropy. 
We study the directional entropy of a system under a low complexity assumption (this assumption is made precise in Theorem~\ref{trichotomy}).

 Boyle and Lind~\cite{BL}  further 
analyzed directional entropy for topological dynamical systems and related it to expansive subdynamics.   
We  use their definition, but restricted to our two dimensional setting: 
\begin{definition}
If $X$ is a dynamical system with a continuous $\Z^2$ action $(T^{\bn}\colon \bn\in\Z^2)$, we say that 
a line $\ell\subset\mathbb{R}^2$ is {\em expansive} if there exists $r> 0$ such that if 
$x,y\in X$ satisfy $x({\bn}) = y(\bn)$ for all $\bn\in\Z^2$ with $\rho(\bn, \ell)< r$, then $x=y$.  
If  $\ell$ is not an expansive line, we say that it is {\em nonexpansive}.  
\end{definition}

For the full shift $X = \CA^{\Z^2}$ with the $\Z^2$ action by translations, it is easy to check that 
there are no expansive lines. 
However, restricting to a system of the form $X_\eta$ for some $\eta\colon\Z^2\to\CA$, there are more possibilities.  
If $\eta$ is periodic, then either the directional entropy $h(\mathbf{u}) = 0$ for all $\mathbf{u} \in S^1$ or there is a single direction of zero entropy.  
In the former case, $\eta$ has an expansive direction with zero entropy and in the latter case, the unique direction of zero entropy 
is nonexpansive.  

Thus, assuming Nivat's conjecture, if there exist $n,k \in \Z$ such that 
$P_\eta(n,k) \leq nk$, then  the directional entropy of $\eta$  is either zero in all directions or there is a unique direction of zero directional entropy.    We show that this conclusion holds under the stronger hypothesis 
that a complexity assumption holds for infinitely many pairs $n_i, k_i$ (as usual, $\mathbf{e}_1$
		and $\mathbf{e}_2$ denote the standard basis vectors): 
\begin{theorem}
\label{trichotomy}
Assume $\CA$ is a finite alphabet and $\eta\colon \Z^2\to\CA$.  If there exists an infinite sequence $n_i, k_i\in\N$ 
such that $P_\eta(n_i,k_i) \le n_ik_i$,
then either
	\begin{enumerate}
		\item $h(\mathbf{u}) = 0$ for all $\mathbf{u} \in S^1$, or
		\item there is a unique nonexpansive direction for $\eta$, which is either $\mathbf{e}_1$
		or $\mathbf{e}_2$, and $\eta$ is periodic in this direction.
	\end{enumerate}
\end{theorem}

An immediate consequence is the following.
\begin{cor}
Assume $\CA$ is a finite alphabet and $\eta\colon \Z^2\to\CA$.  If there exists an infinite sequence $n_i, k_i\in\N$ 
such that  $P_\eta(n_i,k_i) \le n_ik_i$,
	then $\eta$ has zero directional entropy along each of its nonexpansive directions.
\end{cor}

\section{Sufficient conditions for zero directional entropy}
\label{sec:directional}

We start by reviewing some definitions from ~\cite{CK}.  
If $\CS\subset\mathbb{R}^2$, we denote the convex hull of $\CS$ by $\conv(\CS)$.  We say $\CS\subset\Z^2$ is {\em convex} if $\CS=\conv(\CS)\cap\Z^2$.
Define the {\em volume} of convex $\CS\subset\Z^2$ to 
be the volume of its convex hull and define the \emph{boundary} $\partial(\CS)$ of a convex set $\CS\subset \Z^2$
to be the boundary of $\conv(\CS)$.
Given a convex set in $\Z^2$ of positive volume, 
we endow its boundary with the positive orientation, so that
it consists of directed line segments.  If $\CS\subset\Z^2$ is convex and has zero volume, then $\conv(\CS)$ is 
a line segment in $\mathbb{R}^2$ and in this case, we do not define an orientation on $\partial(\CS)$.

For finite $\CS\subset\Z^2$ and $\eta\colon\Z^2\to\CA$, we define $X_\CS(\eta)$ to be the 
$\Z^2$ subshift of finite type generated by the $\CS$ words of $\eta$, 
meaning that $X_\CS(\eta)$ consists of all $f\in\CA^{\Z^2}$ such that all $f$-colorings of $\CS$ 
appear as $\eta$-colorings of $\CS$.  

\begin{defn}
Suppose $\CS\subset\mathcal{T}\subset\Z^2$ are nonempty, finite sets and that $f\colon\CS\to\CA$ is an $\eta$-coloring of $\CS$.  We say that $f$ {\em extends uniquely} to an $\eta$-coloring of $\mathcal{T}$ if there is exactly one $\eta$-coloring of $\mathcal{T}$ whose restriction to $\CS$ coincides with $f$.
\end{defn}

\begin{defn}
If $\CS \subset \Z^2$ is a  nonempty, finite, convex set, 
then $x\in\CS$ is {\em $\eta$-generated} by $\CS$ if every $\eta$-coloring of $\CS\setminus\{x\}$ 
extends uniquely to an $\eta$-coloring of $\CS$, and $\CS$ is an {\em $\eta$-generating set} if every boundary vertex of $\CS$ 
is $\eta$-generated.  When $\eta$ is clear from context, we refer to an $\eta$-generating set as a {\em generating set}.
\end{defn}

Generating sets give rise to zero topological entropy: 
\begin{lemma}[\cite{CK}, Lemma 2.15]
\label{topentropy}
If $\CS \subset \Z^2$ is a generating set for 
$\eta \colon \Z^2\to \mathcal{A}$ and $\CS^{\prime} \supset \CS$ is finite, 
then the topological entropy of the $\Z^2$ dynamical system 
$(X_{\CS^{\prime}}, \{T^{\mathbf{u}}\}_{\mathbf{u}\in\Z^2})$ is zero. 
\end{lemma}

We review the definition of directional entropy introduced by Milnor~\cite{milnor}.

\begin{notation}
\label{notation:larger}
If $T$ is a continuous $\Z^2$ action on the compact metric space $(X,\rho)$, 
$E\subset \mathbb{R}^2$ is a compact set, and $\vre > 0$,
set $N_{T}(E,\vre)$ to be cardinality of the smallest set $Y \subset X$
such that for each $x \in X$ there exists $y \in Y$ with $\rho(T^{\bfn}(x) , T^{\bfn}(y)) < \vre$
for each $\bfn \in E \cap \Z^2$. 
For a compact set $E$ and $t > 0$ , let 
$$
E^{(t)} = \{\bfv\colon \|\bfv - \bfu\| < t \text{ for some } \bfu \in E\}$$
denote the {\em $t$-neighborhood of $E$} and 
let $tE = \{t\bfu\colon \bfu \in E\}$ 
denote the {\em  $t$-dilation} of $E$. 
\end{notation}

\begin{defn}
If $\Phi$ is a set of $k$ linearly independent vectors and $Q_\Phi$ is the parallelepiped spanned
by $\Phi$, then the \emph{$k$-dimensional topological directional entropy} $h_k(\Phi)$ is defined to be 
$$h_k(\Phi) = \lim_{\vre\to0} \sup_{t > 0} \overline{\lim_{s\to \infty}} \frac{\log N_T((sQ_\Phi)^{(t)},\vre)}{s^k}.$$
\end{defn}

We compute the directional entropy for the $\Z^2$ action by translations on the 
space $X_\eta$, in which case it is straightforward to recast the definition in terms of complexity. 
To do so, we make a slight abuse of notation and for a vector $\bfv \in \mathbb{R}^2$, we write $[0,\bfv]$ for $\{\vre \bfv \colon 0 \le \vre \le 1\}$.
As we are interested in one dimensional directional entropy, to simplify notation, when $\Phi = \{\bfu\}$ for some unit vector $\bfu$, we write $h(\bfu)$, instead of $h_1(\Phi)$.

\begin{lemma}
\label{entropycomplex}
Assume $\eta \colon \Z^2 \to \mathcal{A}$.  
If $\bfu$ is a unit vector, then
the ($1$-dimensional) topological directional entropy $h(\bfu)$ of the $\Z^2$ action by translation
on $X_\eta$ in the direction of $\bfu$ is given by 
$$h(\bfu) = \sup_{t>0}\overline{\lim_{s\to\infty}}\frac{\log P_\eta([0,s\bfu]^{(t)})}{s}.$$
\end{lemma}

\begin{proof}
Let $T$ denote the $\Z^2$ action by translation.
Fix $0 < \vre < 1$ and let $M = \lfloor -\log_{2} \vre\rfloor$. 
If $x,y \in X_\eta$ satisfy
\begin{equation}
\label{eq:close}
\rho\bigl(T^{\bfn}(x),T^\bfn(y)\bigr) < \vre \text{ for all } \bfn \in [0,s\bfu]^{(t)}, 
\end{equation}
then $x$ and $y$ agree on $[0,s\bfu]^{(t+M)}$.  Conversely, 
if $x$ and $y$ agree on $[0,s\bfu]^{(t+M+1)}$, 
they satisfy~\eqref{eq:close}. 
Thus,

$$P_\eta([0,s\bfu]^{(t+M)}) \le N_T([0,s\bfu]^{(t)},\vre) \le P_\eta([0,s\bfu]^{(t+M+1)})$$
and so
$$\lim_{M\to\infty} \sup_{t>0}\overline{\lim_{s\to\infty}}\frac{P_\eta([0,s\bfu]^{(t+M)}) }{s}
	\le h(\bfu)
	\le \lim_{M\to\infty} \sup_{t>0}\overline{\lim_{s\to\infty}} \frac{\log P_\eta([0,s\bfu]^{(t+M+1)})}{s}.$$
But since $P_\eta([0,s\bfu]^{(t)})$ is non-decreasing in $t$,
\begin{align*}\sup_{t>0}\overline{\lim_{s\to\infty}}\frac{\log P_\eta([0,s\bfu]^{(t+M)}) }{s}
	& =  \sup_{t>0}\overline{\lim_{s\to\infty}}\frac{\log P_\eta([0,s\bfu]^{(t+M+1)}) }{s} \\
	& = \sup_{t>0}\overline{\lim_{s\to\infty}}\frac{\log P_\eta([0,s\bfu]^{(t)}) }{s}.\hfill\qedhere
	\end{align*}
	\end{proof}

Given a generating set $\CS\subset\Z^2$ for some $\eta\colon\Z^2\to\CA$, 
applying Lemma~\ref{topentropy} with $\CS'=\CS$ provides an upper bound on the entropy of the associated dynamical system $X_\eta$.  
We use Lemma~\ref{entropycomplex} to strengthen this result: 
\begin{prop}
\label{generating-directional}
Assume $\eta \colon \Z^2 \to \mathcal{A}$ has
an $\eta$-generating set and let $X_\eta$ be the associated dynamical system endowed with the 
$\Z^2$ action by translation.  Then there exists $c > 0$ such that
$h(\bfu) < c$ for all unit vectors $\bfu \in \mathbb{R}^2$.
\end{prop}

\begin{proof}
Let $\CS$ be an $\eta$-generating set and let $d = \diam(\CS)$.
Fix a unit vector $\bfu\in\mathbb{R}^2$ and let $t > 0$.
We claim that there exists $C>0$ such that for sufficiently
large $s > 0$, we have $P_\eta([0,s\bfu]^{(t)}) \le C |\mathcal{A}|^{2ds}$.  Once the claim is proven, the proposition follows from Lemma~\ref{entropycomplex}.

Fix $s> d$. Since $P_\eta([0,s\bfu]^{(t)})$ is non-decreasing in $t$, we can assume that 
$t > d$. Then $[0,s\bfu]^{(t)}$ contains a translate of $\CS$.  Let $\bfu^\perp$ be one of the two unit 
vectors perpendicular
to $\bfu$.  Let $\CS'$ be a translate of $\CS$ such that 
exactly one row of 
$\CS'$ in the direction of $\bfu^\perp$
lies outside $[0,s\bfu]^{(t)}$ and such that $\CS'$ touches one of the two border segments
of $[0,s\bfu]^{(t)}$ in the direction of $\bfu$.  Thus $\CS'$ satisfies
$$\CS' \subseteq [0,(s+1)\bfu]^{(t)} \text{ but } \CS' \not\subseteq [0,(s+\vre)\bfu]^{(t)} \text{ for } \vre < 1,
$$
and 
$$\CS' \cap ([0,(s+1)\bfu] - t\bfu^{\perp}) \neq \emptyset.$$

Fix a coloring of $[0,s\bfu]^{(t)} \cup \CS'$. Since $\CS$ is a generating set,
this coloring extends uniquely to a coloring of 
$[0,s\bfu]^{(t)} \cup (\CS' + [0,\bfu^\perp])$.
This in turn extends uniquely to a coloring of
$[0,s\bfu]^{(t)} \cup (\CS' + [0,2\bfu^\perp])$.
Continuing to extend and using that $\diam(\CS) = d$, 
each coloring of $[0,s\bfu]^{(t)} \cup (s\bfu - t\bfu^\perp+[0,d\bfu^\perp] + [0,\bfu])$
extends uniquely to a coloring of 
$[0,(s+1)\bfu]^{(t)} \setminus (s\bfu + t\bfu^\perp+[0,-d\bfu^\perp]+[0,\bfu])$.
It follows that
$$P_\eta([0,(s+1)\bfu]^{(t)}) \le |\mathcal{A}|^{2d}P_\eta([0,s\bfu]^{(t)}),$$
which completes the proof.
\end{proof}

It was shown in~\cite{sinai} that if a $\Z^2$ topological dynamical system has bounded
directional entropy in all directions, then it has zero topological entropy.  Thus the system $X_\eta$ 
generated by $\eta\colon\Z^2\to\CA$ that also an $\eta$-generating system (as in Proposition~\ref{generating-directional}) has zero topological entropy.
The following example shows that a converse to this result fails,
even for a system $X_\eta$ endowed with translations, showing that Proposition~\ref{generating-directional}
strengthens existing results on the entropy of $X_\eta$: 
\begin{example}
Let $\alpha \colon \Z \to \{0,1\}$ with $P_\alpha(n) = 2^n$
and let $A = \{10^n + i^2 \colon i, n \in \N, 1\le i \le n\}$.
Define $\eta \colon \Z^2 \to \{0,1\}$ by  
$$\eta(i,j) =
\begin{cases}
\alpha(i+j) & \text{if }j \in A\\
\alpha(i) & \text{otherwise.}
\end{cases}
$$
Then the topological entropy of the $\Z^2$ action on $X_\eta$ by translations is zero and 
 $h(\bfe_1) = \infty$.
\end{example}
\begin{proof}
Let  $\beta\colon  \Z \to \{0,1\}$ denote the indicator function of the set $A$.  
We first bound the complexity function $P_\beta(k)$.
Fix $k \in \N$.
Given $m \in \Z$, let $I(m) = \{m, m+1, \dots, m+k-1\}$.
Let $n(m)$ be the smallest $n \in \N$ such that
$10^n + i^2 \in I(m)$ for some $1 \le i \le n$, or
take $n(m) = 0$ if no such $n$ exists. If $n(m) > 0$, let 
$i(m)$ be the minimal $1 \le i \le n(m)$ such that
$n(m) + i^2 \in I(m)$, and let $i(m) = 0$ if $n(m) = 0$.
Finally, let $a(m) = \min(A \cap I(m)) - m$.

If $10^{n(m)} > k$, then $10^n + i^2 \not\in I(m)$ for any $n \neq n(m)$
and any $1 \le i \le n$.
Note that if $i = i(m) > k$, then $(i+1)^2 - i^2 \ge i^2 - (i-1)^2 \ge 2i - 1 > k$,
and so $I(m) \cap A = \{10^{n(m)} + i(m)\}$.
Thus, $\beta\rst{I(m)}$ is determined by $0 \le \min\{n(m), k\} \le k$,
 $0 \le \min\{i(m), k\} \le k$, and $0 \le a(m) \le k$, and
 so $P_\beta(k) \le (k+1)^3$.
 Since $I(m)$ contains at most $\sqrt{k}\log_{10} k \le k^{3/4}$ elements of $A$, 
$$P_\eta(n,k) \le P_\beta(k) (2^n)^{k^{3/4}} = k^32^{nk^{3/4}}$$
and so 
$$\lim_{n\to\infty}\frac{\log P_\eta(n,n)}{n^2} 
		\leq \lim_{n\to\infty}\frac{3\log n + n^{7/4}\log 2}{n^2} = 0.$$

But since there are exactly $\lfloor \sqrt{k}\rfloor$
	elements of $A$ in $10^n + [1,k]$ when $k \le n$, we have that
$$\frac{\log P_\eta([0, n\bfe_1]^{(k/2)})}{n} \ge \frac{\log P_\eta(n,k)}{n}
	\ge \frac{\log (2^n)^{\sqrt{k-1}}}{n} =\sqrt{k-1}\log 2.$$
By Lemma~\ref{entropycomplex}, it follows that 
\begin{equation*}
h(\bfe_1) \ge \sup_{k>0}\sqrt{k-1}\log 2= \infty.  \hfill\qedhere
\end{equation*}
\end{proof}

As shown in~\cite{CK}, if there exist $n, k \in \N$ such that $P_\eta(n,k) \le nk$,
then there is an $\eta$-generating set.  Thus for $\eta$ satisfying such a complexity bound, 
the $\Z^2$ action by translations on $X_\eta$ has bounded directional entropy in all directions.  
In particular, it has zero topological entropy.

\section{A sequence of complexity bounds}

\begin{notation}
\label{notation:thick}
For a unit vector $\bfu \in \mathbb{R}^2$ and a compact set $K\subset\mathbb{R}^2$, we let  
$\tau_\bfu(K)$ denote the thickness of the compact set $K$ in the direction of $\bfu$, 
defined by 
$$\tau_\bfu(K) = \sup\{ \tau \colon [0,\tau \bfu] + \bfn \subset K\text{ for some }\bfn \in \Z^2\}.$$
\end{notation}

\begin{prop}
\label{convexentropy}
Assume $\eta\colon\Z^2\to\CA$.  
The $\Z^2$ action by translation on $X_\eta$ has zero entropy 
in direction $\bfu\in\mathbb{R}^2$ if and only if there exist compact sets $K_i \subset \mathbb{R}^2$
such that $\lim_{i \to \infty} \frac{\log P_\eta(K_i)}{\tau_\bfu(K_i)} = 0$.
\end{prop}

\begin{proof}
Assume there exists such a sequence $K_i$ and assume for contradiction that
$h(\bfu) = 4\delta > 0$.  
Since $P_\eta([0,s\bfu]^{(t)})$ is non-decreasing in $t$ (recall Notation~\ref{notation:larger}), using  Lemma~\ref{entropycomplex}, 
there exists $t_0 > 0$ such that
whenever $t \ge t_0$,
$$\overline{\lim_{s \to \infty}} \frac{\log P_\eta([0,s\bfu]^{(t)})}{s} \ge 3\delta.$$
Thus there exists a sequence $(s_m)$ such that
$\log P_\eta([0,s_m\bfu]^{(t_0)})\ge 2\delta s_m$.
Set $\tau_i = \tau_\bfu(K_i)$.
If $s_m \le s \le s_m+\tau_i$ and $s_m \ge \tau_i$, then
$$\log P_\eta([0,s\bfu]^{(t_0)})\ge \log P_\eta ([0,s_m\bfu]^{(t_0)}) \ge 2\delta s_m \ge 2\delta s \frac{s_m}{s}
 \ge 2\delta s \frac{s_m}{s_m + \tau_i} \ge \delta s.$$
Hence, there exist infinitely many $j \in \N$ such that for all $t \ge t_0$,
$$\log P_\eta ([0,j\tau_i\bfu]^{(t)}) \ge \delta j \tau_i.$$
But since $K_i$ contains a translate of $[0,\tau_i\bfu]$, it follows that
$$\log P_\eta(K_i)^{2t_0j} \ge \log P_\eta ([0,j\tau_i\bfu]^{(t_0)}) \ge \delta j \tau_i,$$
and so
$$\frac{\log P_\eta(K_i)}{\tau_i} \ge \frac{\delta}{2t_0}$$
for all $i \in \N$, a contradiction.

Conversely, if no such sequence exists, then setting $K_i = [0,i\bfu]^{(1)}$,
there exists a constant $c > 0$ such that $\log P_\eta(K_{i_j}) \ge ci_j$ for some increasing sequence
$(i_j)$.  By Lemma~\ref{entropycomplex}, $h(\bfu) \ge c$.
\end{proof}

\begin{cor}
\label{exponential-eccentricity}
For $\eta\colon\Z^2\to\CA$, if there exist $n_i, k_i\in\N$ tending to infinity such that $P_\eta(n_i,k_i) \le n_ik_i$ and
$\lim \frac{\log (n_i)}{k_i} = \lim \frac{\log (k_i)}{n_i} = 0$,
then the $\Z^2$ action by translation on $X_\eta$ has zero entropy in all directions.
\end{cor}

\begin{proof}
We apply Proposition~\ref{convexentropy} to the sets $K_i = [0,n_i-1]\times [0,k_i-1]$.
If $\mathbf{u}$ is a unit vector, then $\tau_{\mathbf{u}}(K_i) \ge \min(n_i, k_i)$.
Without loss of generality, we can assume that $k_i \ge n_i$ for all $i \in \N$.
Then
$$\lim_{i \to \infty} \frac{\log P_\eta(K_i)}{\tau_{\mathbf{u}}(K_i)} \le \frac{\log(n_i k_i)}{n_i}
		\le 2\frac{\log(k_i)}{n_i} = 0.$$
\end{proof}

\begin{remark}
\label{exponential-eccentricity-remark} 
By passing to a subsequence, the conclusion of Corollary~\ref{exponential-eccentricity} holds unless 
the rectangles for which we have complexity assumptions have eccentricity unbounded either above or below.  In particular, it holds unless 
	there exists $C > 1$ such that either $k_i \ge C^{n_i}$ for all $i \in \N$ or 
	$n_i \ge C^{k_i}$ for all $i \in \N$, and this is the setting studied in the next section.  
\end{remark}

\section{Proof of Theorem~\ref{trichotomy}}
We say two vectors $\mathbf{v}, \mathbf{w} \in \mathbb{R}^2 \setminus \{0\}$ are \emph{parallel}
if $\mathbf{v} = c\mathbf{w}$ for some $c > 0$, and we say they are
\emph{antiparallel} if $\mathbf{v} =c\mathbf{w}$ for some $c < 0$.
We define these terms analogously for directed lines and line segments.

Recall that we endow the boundary of a convex set $S\subset \Z^2$ with positive orientation.  
Given $\mathbf{v}\in \mathbb{R}^2 \setminus \{0\}$, a $\mathbf{v}$-plane is a closed half-plane whose boundary
is parallel to $\mathbf{v}$. For example, $\{(x,y) \colon x \le 2\}$ is an $\mathbf{e}_2$-plane,
while $\{(x,y) \colon x \ge 2\}$ is a $(-\mathbf{e}_2)$-plane. 
\begin{notation}
If $\CS\subset\Z^2$ is convex
and $\ell$ is a directed line, we write $E(\ell, \CS) = \ell' \cap \CS$, where
$\ell'$ is the boundary of the intersection of all $\ell$-planes containing $\CS$.
\end{notation}
Note that $E(\ell, \CS)$ is the set of integer points on some edge of $\partial(\CS)$, and it
may reduce to a single vertex.

We recall a definition from~\cite{CK}: 
\begin{definition}
\label{def:balanced}
Suppose $\ell\subset\mathbb{R}^2$ is a directed line.
A finite, convex 
set $\CS\subset\Z^2$ is {\em $\ell$-balanced for $\eta$} if
\begin{enumerate}
\item The endpoints of $E(\ell, \CS)$ are $\eta$-generated by $\CS$; 
\item The set $\CS$ satisfies $P_{\eta}(\CS\setminus E(\ell, \CS)) >
			P_{\eta}(\CS)-|E(\ell, \CS)|$;
			\label{item:two}
\item Every line parallel to $\ell$ that has nonempty intersection with 
$\CS$ intersects $\CS$ in at least $|E(\ell, S)|-1$ integer points.
\end{enumerate}
We also call such a set \emph{$\mathbf{v}$-balanced for $\eta$}
whenever $\mathbf{v}$ is a vector $\mathbf{v}$ parallel to $\ell$.
\end{definition}

We note that the endpoints of $E(\ell, \CS)$ could consist of a single endpoint.  Furthermore, 
$P_\eta(\emptyset)$ is not defined and so a generating set $\CS$ does not consist of a line segment, 
meaning that $\CS\neq E(\ell, \CS)$ and so condition~\eqref{item:two} is well defined.  

\begin{lemma}
\label{balanced}
Suppose $\eta\colon \Z^2 \to \CA$ satisfies $P_\eta(n_1,k_1) \le n_1k_1$ for some $n_1, k_1 \in \N$.
If $\ell$ is a horizontal or vertical directed line, then there is an $\ell$-balanced set for $\eta$.
Furthermore, there exist $\mathbf{e}_1$ and $(-\mathbf{e}_1)$-balanced
sets with the same height, and there exist $\mathbf{e}_2$- and $(-\mathbf{e}_2)$-balanced
sets with the same width.
\end{lemma}

This is essentially proved in Lemma 4.7 in~\cite{CK}.  There, the assumption that there
exist $n, k$ with $P_\eta(n,k) \le nk$ is replaced with the stronger assumption that
we can find $n, k$ with $P_\eta(n,k) \le \dfrac{nk}{2}$, but this hypothesis is not necessary
in the case of a horizontal or vertical line,  the only ones needed here.
We include a proof for completeness.

\begin{proof}
We prove the case that $\ell$ is parallel to $\mathbf{e}_2$;
the other three cases are handled similarly.
Let $n_1' \le n_1$ be the minimal positive integer such that $P_\eta(n_1', k_1) \le n_1'k_1$. First suppose $n_1'\leq 1$. Then for each $n \in \Z$,
$\eta \rst {\{n\}\times \Z}$ is vertically periodic with period at most $k_1$
by the Morse-Hedlund Theorem \cite{MH}, and so $\eta$ is vertically periodic.
Let $p$ denote its period.

Set $c = P_\eta(\{0\}\times [0,p-1])$ and take $k > \max\{p,  c^2 - c\}$.
Define $\CS = [0,1]\times [0,k-1]$.
Then $E(\ell, \CS) = \{1\}\times [0,k-1]$ and property (iii) of Definition~\ref{def:balanced} 
is clearly satisfied. Property (i) is satisfied since $k > p$.
Finally, property (ii) follows since
$$P_\eta(\CS) \le c^2 < c + k = P_\eta(\CS \setminus E(\ell, \CS)) + |E(\ell, \CS)|.$$ 
Note that in this case, $\CS$ is both an $\mathbf{e}_2$-balanced set
and a $(-\mathbf{e}_2)$-balanced set.
Hence, we may assume that $n_1'\geq2$.  

Set $R_1 = [0, n_1'-2] \times [0, k_1-1]$ 
and set $R_2 = [0, n_1' - 1] \times [0, k_1-1]$.
Note that $R_1$ is simply the result of removing the rightmost vertical line from $R_2$.
By the minimality of $n_1'$, we have 
\begin{equation}
\label{edgeremoval}
P_\eta(R_1) = P_\eta(n_1' - 1, k_1) > (n_1'-1)k_1 \ge P_\eta(R_2) - k_1.
\end{equation}
Take $\CS \subset \Z^2$ 
to be a convex set of minimal size with $R_1 \subsetneq \CS \subset R_2$
satisfying $P_\eta(\CS) - |\CS| = P_\eta(R_2) - |R_2|$.
Then by the minimality of $n_1'$, $\CS$ contains at least one point on the 
vertical line $x = n_1'-1$,
and so $E(\ell, \CS) \subset \{(x,y) \colon x = n_1'-1\}$.
By the minimality of $\CS$,
removing an endpoint of $E(\ell, \CS)$ results in a set $\CS'$ with
$P_\eta(\CS') - |\CS'| > P_\eta(\CS) - |\CS| = P_\eta(\CS) - |\CS'| - 1$.
Thus, since $\CS' \subset \CS$, $P_\eta(\CS') \le P_\eta(\CS) \le P_\eta(\CS')$,
which implies that every $\eta$-coloring of $\CS'$ extends uniquely to an $\eta$-coloring of $\CS$ (or, in other words, the endpoint we removed was $\eta$-generated by $\CS$).
Hence, property (i) in Definition~\ref{def:balanced} holds. 

By construction, we have $$P_\eta(\CS) - |E(\ell, \CS)|
		= P_\eta(\CS) - |\CS| + |R_1|
		= P_\eta(R_2) - |R_2| + |R_1|
		= P_\eta(R_2) - k_1.$$
By~\eqref{edgeremoval}, it follows that
$P_\eta(\CS) - |E(\ell, \CS)| < P_\eta(R_1) = P_\eta(\CS \setminus E(\ell, \CS))$,
which is property (ii) of Definition~\ref{def:balanced}.
Finally, any vertical line other than $x=n_1'-1$ that intersects $\CS$
meets that set in $k_1 \ge |E(\ell, \CS)|$ points, and so property (iii) 
of Definition~\ref{def:balanced} holds as well,
and $\CS$ is $\mathbf{e}_2$-balanced for $\eta$.
Finally, notice that the width of $\CS$ is
$n_1'$ and the choice of this $n_1'$ would be the same if the argument is repeated 
to produce a $(-\mathbf{e}_2)$-balanced set.
\end{proof}

\begin{definition}
Given a directed line $\ell$, an $\ell$-balanced set $\CS$ for $\eta$, an integer $p \ge 0$, 
and a convex set 
$\mathcal{T} \subset \Z^2$ that contains a translate of $\CS$ and has an edge parallel to $\ell$, 
we define the \emph{$(\ell, \CS)$-extension} of $\CT$ to be
$$\ext_{\ell,\CS}^1(\mathcal{T}) = \mathcal{T} \cup \bigcup_{\mathbf{j} \in J(\CT, \ell, \CS)} 
	(\CS + \mathbf{j}),$$	where 
		$$J_{\CT, \ell, \CS} = \{\mathbf{j} \colon (\CS + \mathbf{j}) \setminus \CT
			= E(\ell, \CS) + \mathbf{j}\}.$$
Note that $J_{\CT, \ell, \CS}$ is the set of integer points on some line segment parallel to $\ell$.
Given an integer $p \ge 0$, let $J_{\CT, \ell, \CS, p}$ be the result of removing the $p$
integer points nearest each endpoint of $J_{\CT, \ell, \CS}$ and define
$$\ext_{\ell,\CS, p}^1(\mathcal{T}) = \mathcal{T} \cup \bigcup_{\mathbf{j} \in J_{\CT, \ell, \CS, p}} 
	(\CS + \mathbf{j}).$$
For each $n\geq 1$, we then inductively define 
$$\ext_{\ell,\CS, p}^{n+1}(\mathcal{T}) = \ext_{\ell, \CS, p}^1(\ext_{\ell, \CS, p}^n(\mathcal{T})).$$
Note that $\ext_{\ell,\CS}^{n}(\mathcal{T}) = \ext_{\ell,\CS,0}^{n}(\mathcal{T})$.
For notational convenience, we set $\ext_{\ell, \CS, p}^0(\mathcal{T}) = \mathcal{T}$.

We define the \emph{$(\ell, \CS)$-border} $\partial_{\ell, \CS}(\mathcal{T})$
of $\mathcal{T}$
to be  $\bigcup_{\mathbf{j} \in J_{\CT, \ell, \CS}}(\tilde{\CS} + \mathbf{j})$,
where $\tilde{\CS} = \CS \setminus E(\ell, \CS)$.   We define the 
\emph{$(\ell,\CS,p)$-border} $\partial_{\ell,\CS,p}(\CT)$ of $\CT$ to b e $\bigcup_{\mathbf{j} \in J_{\CT, \ell, \CS,p}}(\tilde{\CS} + \mathbf{j})$

Given two sets $\CS, R \subset \Z^2$, 
we say $f \colon R \to \CA$ is an \emph{$(\CS,\eta)$-coloring} of $R$
if $f = g\rst{R}$ for some $g \colon \Z^2 \to \CA$ such that $g\rst{{\CS+ \mathbf{j}}}$
is an $\eta$-coloring of $\CS$ for each $\mathbf{j} \in \Z^2$.
\end{definition}

Note that every $\eta$-coloring of $R$ is also an $(\CS,\eta)$-coloring of $R$,
but the converse does not always hold.

To prove the next lemma, we recall a finite version of the Morse-Hedlund Theorem.

\begin{definition}
If $a\in\Z$ and $f\colon\{a,a+1,a+2,\dots,a+i-1\}\to\CA$, define $Tf\colon\{a-1,a,\dots,a+i-2\}\to\CA$ by $(Tf)(n):=f(n+1)$ and define $P_f(n)$ to be the number of distinct functions of the form $(T^mf)\rst{\{a,a+1,\dots,a+n-1\}}$, where $0\leq m\leq i-n$ and $0\leq n\leq i$.
\end{definition}

The following is essentially due to Morse and Hedlund~\cite{MH}, and appears with this formulation in~\cite{CK2}:  
\begin{theorem}
\label{MorseHedlundTheorem}
Suppose $f\colon\{a,a+1,\dots,a+i-1\}\to\CA$ and suppose there exists $n_0\in\N$ such that $P_f(n_0)\leq n_0$.
If $i>3n_0$, then the restriction of $f$ to the set $\{a+n_0,a+n_0+1,\dots,a+i-n_0\}$ is periodic of period at most $n_0$. 
\end{theorem}

\begin{lemma}
\label{nonunique-periodic}
Let $\ell$ be a vertical  (respectively, horizontal) directed line, 
and assume $\eta \colon \Z^2 \to \CA$ has an $\ell$-balanced set $\CS$.
Suppose $R$ is a rectangle
(with horizontal and vertical sides) large enough
to contain $\CS$ and define $q:=|\CS|$.  If $f$ is an $\eta$-coloring of $\ext_{\ell, \CS}^1(R)$ 
such that $f\rst{R}$ does not extend uniquely to an $(\CS,\eta)$-coloring of $\ext_{\ell, \CS}^1(R)$, 
then $f$ is vertically (respectively, horizontally) periodic on $\partial_{\ell,\CS,2q}(R)$
with period at most $|E(\ell, \CS)| - 1$.
\end{lemma}

\begin{proof}
For convenience, set $R = [0,n-1] \times [0, k-1]$ and $R^1 = \ext_{\ell, \CS}^1(R)$.
We prove the claim when $\ell$ is parallel to $\mathbf{e}_2$;
the other cases are similar.
In this case, $R^1 \setminus R = \{n\} \times I$
for some interval $I$.
Write $I = \{j_0, j_0 + 1, \dots, j_0 + L - 1\}$, where $k - q \le L \le k$.
We may assume without loss of generality that $I$ is also the set of integers $j$
such that 
$(\CS+ (0,j)) \setminus R = E(\ell, \CS) + (0,j)$.
(Otherwise replace $\CS$ with a translate that has this property.)
Let $\tilde{\CS} = \CS \setminus E(\ell, \CS)$.
The assumptions imply that for each $j \in I$,
$f\rst{{\tilde{\CS}+(0,j)}}$ does not extend uniquely to 
an $\eta$-coloring of $\CS + (0,j)$. Indeed, if it did
extend uniquely, then since the endpoints of $E(\ell, \CS)$ are $\eta$-generated,
$f\rst{R}$ would extend uniquely to an $(\CS,\eta)$-coloring of $R^1$.

Since $\CS$ is $\ell$-balanced, 
$P_\eta(\tilde{\CS}) > P_\eta(\CS) - |E(\ell, \CS)|$.  
Thus there are at most $|E(\ell, \CS)|-1$ distinct $\eta$-colorings of $\tilde{\CS}$ that do not extend
uniquely to an $\eta$-coloring of $\CS$. Thus, at most $|E(\ell, \CS)|-1$ such colorings
appear as a coloring of the form $f\rst{{\tilde{\CS} + (0,j)}}$ for $j \in I$. 

Set $\tilde{\CS_0} = \{(i,j) \in \tilde{\CS} \colon (i,j-1) \not\in \tilde{\CS}\}$.
Since $\CS$ is a balanced set, $\tilde{\CS_0} + (0,j) \subset \tilde{\CS}$ for each 
$0 \le j \le |E(\ell, \CS)| - 2$.
Let $\mathcal{B}$ be the set of $\CA$-colorings of $\tilde{\CS_0}$
and define $g \colon I \to \mathcal{B}$
by $g(j) = f\rst{{\tilde{\CS_0} + (0,j)}}$. Then the one dimensional complexity 
$P_g(|E(\ell, \CS)|-1)$
is bounded above by the number of colorings of 
$\tilde{\CS} \supset \bigcup_{0 \le j \le |E(\ell, \CS)| - 2} (\tilde{\CS_0} + (0,j))$
that arise as a coloring of the form $f\rst{\tilde{\CS} + (0,j)}$ with $j \in I$, and so  
$P_g(|E(\ell, \CS)| -1) \le |E(\ell, \CS)|-1$.
Hence by Theorem~\ref{MorseHedlundTheorem}, 
$g$ is periodic on 
$$\{i_0 + |E(\ell, \CS)| -1, i_0 + |E(\ell, \CS)| -1, \dots, i_0 + L - (|E(\ell, \CS)| - 1) \}$$
with period at most $|E(\ell, \CS)| - 1$.
Since $|E(\ell, \CS)| - 1 < q$ and $L \ge k - q$, this implies that
$f\rst{\partial_{\ell,\CS, 2q}(R)}$ is vertically periodic with period at most $|E(\ell, \CS)|-1$.
\end{proof}

\begin{lemma}
\label{extend-periodicity}
Let $\ell$ be a vertical (respectively,  horizontal) directed line and let
$\eta$, $\CS$, $q$, and $R$ be as in Lemma~\ref{nonunique-periodic}.
Let $p\in \N$ and let $a = \max\{p, 2q\}$.
Let $m \in \N$ and
let $f$ be an $\eta$-coloring of $R$ which is vertically (respectively,  horizontally)
periodic on $\partial_{\ell, \CS}(R)$ with period at most $p$.
Then any extension of $f$ to an $\eta$-coloring of $\ext_{\ell, \CS, a}^m(R)$
must be vertically (respectively, horizontally) periodic on 
$\ext_{\ell,\CS, a}^m(R) \setminus \ext_{\ell, \CS, a}^{m-1}(R)$
with period at most $p$, and therefore vertically (respectively, horizontally) periodic
on $\ext_{\ell, \CS, a}^m(R)$ with period at most $p!$.
\end{lemma}

We note that the proof is similar in spirit to the proof of Proposition 4.8 in \cite{CK}.

\begin{proof}
We prove this when $\ell$ is parallel to $\mathbf{e}_2$, by induction on $m$. 
The other cases
are proved similarly.

Let $T_0 = R = [0,n-1] \times [0,k-1]$, and for $m \ge 1$ let
$T_{m} = \ext_{\ell, \CS, a}^m(R)$.
Let $\tilde{f}$ be an extension of $f$ to an $\eta$-coloring
of $T_{m+1}$ and
suppose that the restriction of $\tilde{f}$ to $\partial_{\ell,\CS}(T_m)$ is periodic of period 
$p' \le p$.
Then this $\eta$-coloring either extends uniquely to an $(\CS,\eta)$-coloring of 
$T_{m+1}$
or it does not. We claim that in either case, $f\rst{\partial_{\ell, \CS}(T_{m+1})}$ is vertically
periodic of period at most $p$.

{\bf Case 1:} \emph{$f\rst{T_m}$ does extend uniquely to 
		an $(\CS,\eta)$-coloring of $T_{m+1}$.}
		Since $\tilde{f}$ is an $\eta$-coloring of $T_{m+1}$, there
		exists $(i_0,j_0)$ such that 
		for all $(i,j) \in T_{m+1}$, $\tilde{f}(i,j) = \eta(i + i_0, j + j_0)$.
		Define $g \colon T_{m+1} \to \mathcal{A}$ by
		$$g(i,j) = \begin{cases}
				\eta(i+i_0,j+j_0) &\mbox{ if } (x,y) \in T_m\\
				\eta(i + i_0,j+j_0 + p') &\mbox{ if } (x,y) \in T_{m+1}\setminus T_m.
				\end{cases}$$
		We claim that $g$ is an extension of $f\rst{T_m}$
		to an $(\CS,\eta)$-coloring
		of $T_{m+1}$.
		To see this, fix $(i_1, j_1) \in \Z^2$ such that $\CS + (i_1, j_1)$ intersects
		$T_{m+1}$. If this translate does not intersect $T_{m+1} \setminus T_m$,
		then $g\rst{\CS + (i_1,j_1)}  = \tilde{f}\rst{\CS+(i_1,j_1)}$ and
		so $g\rst{\CS + (i_1,j_1)}$ is (or can be extended to) an $\eta$-coloring of $\CS$. 
		If $\CS+(i_1,j_1)$ intersects $T_{m+1} \setminus T_m$, then
		by the definition of $T_{m+1}$,
		we must have $(\CS+(i_1,j_1+p')) \cap [0, n+m-1] \times \Z \subset T_m$.
		Since $\tilde{f}$ is vertically periodic on $\partial_{\ell, \CS}(T_m)$
		with period $p'$, we see that $g(i,j) = \eta(i+i_0,j+j_0+p')$ for 
		$(i,j) \in (\CS + (i_1,j_1)) \cap (R_m \cup R_{m+1})$.
		Thus, $g$ is an $(\CS, \eta)$-coloring of $T_m \cup T_{m+1}$ which restricts to $f$ on $T_m$.  
		By assumption we must have
		$g= \tilde{f}$, meaning that 
		$\tilde{f}$ is vertically periodic on $T_{m+1} \setminus T_m$
		with period dividing $p'$, and therefore periodic on $\partial_{\ell, \CS}(T_{m+1})$
		with period $p' \le p$.
	 
{\bf Case 2:} \emph{$f\rst{T_m}$ does not extend uniquely to 
		an $(\CS,\eta)$-coloring of $T_{m+1}$.}
		Let $\tilde{\CS} = \CS \setminus E(\ell, \CS)$.
		By Lemma~\ref{nonunique-periodic}, 
	$f\rst{\partial_{\ell,\CS, 2q}(T_m)}$ is vertically periodic 
	of period at most $h = |E(\ell, \CS)\cap \CS|-1$.
	Let $N$ be the number of colorings $\alpha$ of $\CS$ 
	such that $\alpha\rst{\tilde{\CS}}$ extends in more than one way to an $\eta$-coloring
	of $\CS$. We claim that $N \le 2h$.
	For each such $\alpha$, let $C_\alpha$ be the $\eta$-colorings $\alpha^{\prime}$ of 
	$\CS$ such that $\alpha\rst{\tilde{\CS}} = \alpha'\rst{\tilde{\CS}}$.
	Since $\CS$ is $\mathbf{e}_2$-balanced for $\eta$,
	$$P_\eta(\tilde{\CS}) + h \ge 
	P_\eta(\CS) = P_\eta(\tilde{\CS}) + \sum_{C_\alpha} (|C_\alpha| - 1).$$
	In particular, $\alpha\rst{\tilde{\CS}}$ extends in more than one way exactly when
	$|C_\alpha| > 1$.  
	
	Enumerating the colorings of $\tilde{\CS}$ that extend in more than one way to 
	a coloring of $\CS$ as $\alpha_1, \dots, \alpha_r$
	(where $r \le h$), we have that 
	$$N = \sum_{i=1}^r |C_{\alpha'_i}| 
	= \sum_{i=1}^r (C_{\alpha'_i} - 1) + \sum_{i=1}^r 1
	\le h + r \le 2h,$$
	where $\alpha'_i$ is a choice of a coloring of $\CS$ that restricts to $\alpha_i$ on $\tilde{\CS}$.  
Without loss of generality, we can assume that $\CS \setminus T_m = E(\ell, \CS)$, 
	but $(\CS + (0,-1)) \setminus T_m \neq E(\ell, \CS) + (0,-1)$. (If not we can replace 
	$\CS$ with a translate that has this property and it continues to be an 
	$\mathbf{e}_2$-balanced set for $\eta$.)
	Then by the pigeonhole principle, there exist integers $0 \le i < j < 2h$
	such that $f\rst{\CS + (0,i)} = f\rst{{\CS}+(0,j)}$.
	Since $f\rst{\partial_{\ell, \CS, 2q}(T_m)}$ is vertically periodic of period at most $h$ and
	each vertical line intersecting $\CS$ intersects it in at least $h$ points
	(since $\CS$ is a vertically balanced set), it follows that
	$f\rst{\tilde{\CS} + (0,i+k)} = f\rst{\tilde{\CS} + (0,j+k)}$ for each $k$ such that
	$(\tilde{\CS} + (0,i+k))  \cup (\tilde{\CS} + (0,j+k)) \subset \partial_{\ell, \CS, 2q}(T_m)$. 
	But since each endpoint of
	$E(\ell, \CS)$ is generated, an easy induction argument shows that 
	$f\rst{\CS + (0,i+k)} = f\rst{\CS + (0,j+k)}$.  Thus $f\rst{\partial_{\ell,\CS}(T_{m+1})}$ is vertically
	periodic of period at most $2h \le p$.
\end{proof}

\begin{lemma}
\label{few-periodic}
If $\eta\colon\Z^2\to\CA$ is not vertically periodic
and there exists an infinite sequence $n_i, k_i \in \N$ such that $P_\eta(n_i, k_i) \le n_i k_i$,
then for any $p \in \N$ and $\lambda > 1$,
there exists $h \in \N$ such that for sufficiently large $w \in \N$,
the number of $\eta$-colorings of $[0,w-1] \times [0,h-1]$
that are vertically periodic with period at most $p$ is less than $\lambda^w$.
\end{lemma}

The proof requires the following technical lemma: 
\begin{lemma}
\label{period-break}
Let $\eta$, $n_i$, $k_i$, $w$ and $p$ be as in Lemma~\ref{few-periodic}.
Let $a = \max\{p, 2n_1k_1\}$, 
set $n_i' = \left\lfloor n_i/3\right\rfloor$ for $i\in\N$, $h_i = 2(p+k_i)$,
$R_{i,w} = [0, w -1] \times [0,h_i- 1]$, 
and $S_i = [0, n_i' -1] \times [0, h_i-1]$.

There exists a constant $C$ independent of $w$ and $i$
such that 
for any $i \in \N$ with $k_i > 4p$,
there exist $\eta$-colorings $g_1, \ldots, g_C$ of $S_i$
such that if
$n_i \le w$, $(x_0, y_0) \in \Z^2$, and 
$\eta\rst{R_{i,w} + (x_0, y_0)}$ is vertically periodic with period $p$,
then the following hold:
	\begin{enumerate}[label=(\alph*)]
		\item\label{it:a1}
		Either there exists minimal $y_1 \ge y_0 + h$ such that for some 
		$x_1 \in [x_0, x_0 + w - 1]$,
				$\eta(x_1,y_1) \neq \eta(x_1,y_1-p)$, or
		\item	\label{it:a2} there exists maximal $y_1 < y_0$ such that 
		$\eta(x_1,y_1) \neq \eta(x_1,y_1+p)$ for some $x_1 \in [x_0, x_0 + w - 1]$,
	\end{enumerate}
	and exactly one of the following holds: 
	\begin{enumerate}[label=(\roman*)]
		\item	\label{it:b1}
				$x_1$ can be chosen to lie in 
		$[x_0 + n_1, x_0 + w - n_1 -1 ]$,
			in which case 
			$\eta$ is horizontally periodic on $[x_0 + a(p+2), x_0 + w - a(p+2)] 
				\times [0,h-1]$
			with period at most $(2n_1)!$,
		\item	\label{it:b2}
				$x_1$ cannot be chosen to lie in $[x_0 + n_1, x_0 + w - n_1 -1 ]$
		but can be chosen to lie in $[x_0, x_0 + n_1-1]$,
		in which case $\eta\rst{S_i + (x_0, y_0)} = g_j$ for some $1 \le j \le C$, or
		\item	\label{it:b3}
				$x_1$ can only be chosen to lie in $[x_0 + w - n_1, x_0 + w - 1]$,
		in which case $\eta\rst{S_i + (x_0 + w - n_i', y_0)} = g_j$ for some $1 \le j \le C$.	
	\end{enumerate}
\end{lemma}

Before proving Lemma~\ref{period-break}, we first use it to prove Lemma~\ref{few-periodic}.

\begin{proof}[Proof of Lemma~\ref{few-periodic}]
For convenience, write $R = R_{i,w}$, $S = S_{i}$, and $h = h_i$.
Let $f$ be an $\eta$-coloring of $R$ that is vertically periodic with period $p' \le p$,
and let $x_0, y_0$ be such that $f(x,y) = \eta(x + x_0, y+ y_0)$
for all $(x,y) \in R$.
Define a finite sequence of rectangles in the following way.
Let $R_0 = R = R_{i,w}$.
For each $0 \le j$ until the process terminates, apply Lemma~\ref{period-break} to $R_j + (x_0,y_0)$.
If case~\ref{it:b1} holds, terminate the process.
If case~\ref{it:b2} holds, let $R_{j+1}'$ be the translate of $S_{i}$ that shares
a left edge with $R_j$ and let $R_{j+1} = R_j \setminus R_{j+1}'$,
which is a rectangle to which the claim also applies, so long as $w - (j+1)n_i' \ge n_i$.
(If this inequality fails, terminate the process instead.)
If case~\ref{it:b3} holds, let $R_{j+1}'$ be the translate of $S_{i}$ that shares a right edge
with $R_{j}$ and let $R_{j+1} = R_j \setminus R_{j+1}'$, which is also a rectangle
to which the claim applies for $w - (j+1)n_i' \ge n_i$.
The coloring $f$ is completely determined by the following data:
\begin{itemize}
	\item The length $m$ of the sequence  of rectangles, which satisfies $m \le \left\lfloor \dfrac{w-n_i}{n_i'}\right\rfloor + 1$.
	\item Whether $R_{j+1}'$ is on the right or left side of $R_j$ for each $0 \le j < m$.
	\item The indices $1 \le a_j \le C$ for which $\eta\rst{R_{j}'+(x_0,y_0)} = g_{a_j}$
		for $1 \le j \le m$.
	\item The restriction of $\eta$ to $R_m+(x_0,y_0)$.
\end{itemize}

Since $\left\lfloor \dfrac{w-n_i}{n_i'}\right\rfloor + 1 \le 4w/n_i$, 
the number of colorings $f$ with these properties is at most
$$\frac{4w}{n_i}2^{4w/n_i}C^{4w/n_i} \max\{C_1, C_2\},$$
where $C_1$ is the number of $\eta$-colorings of $R$ that are horizontally periodic
on $[x_0 + a(p+2), x_0 + w - a(p+2)]\times [0,h-1]$ with period at most $(2n_1)!$
and vertically periodic on $R$ with period at most $p$
and $C_2$ is the number of $\eta$-colorings of $[0, n_i -1]\times [0, h-1]$.
Clearly we have
$$C_1 \le p(2n_1)!|\CA|^{p(2n_1)! + a(p+2)h}$$
and
$$C_2 \le |\CA|^{n_ih}.$$
In particular, 
$C_1$ and $C_2$ are independent of $w$, and so 
the number of $\eta$-colorings of $R$ that are vertically periodic with period
at most $p$ is at most
$K_iw(2C)^{4w/n_i}$,
where $K_i$ is independent of $w$ and $C$ is independent of both $w$ and $i$.
Choose $i$ large enough such that $(2C)^{4/n_i} < \sqrt{\lambda}$.
Then for large enough $w$, the number of colorings $f$ of $R$ that
are vertically periodic with period at most $p$ is less than
\begin{equation*}
K_iw(2C)^{4w/n_i} < K_iw\lambda^{w/2} < \lambda^w. \hfill\qedhere
\end{equation*}
\end{proof}

\begin{proof}[Proof of Lemma~\ref{period-break}]
If $\eta$ is vertically periodic on some strip of width $n_1$, 
then by Lemma~\ref{extend-periodicity} it is periodic on all such strips, 
with bounded period, and so $\eta$ is vertically periodic. 
Hence, we may assume that $\eta$ is not periodic 
on any vertical strip of width $n_1$, meaning that either~\ref{it:a1} or~\ref{it:a2} holds. We  assume throughout the rest of the proof
that~\ref{it:a1} holds; the argument in the other case is similar.

Let $p' \le p$ be the vertical period of $\eta\rst{R_{i,w}+(x_0,y_0)}$,
and again for convenience set $R = R_{i,w}$, $S = S_i$, and $h = h_i$.
Let $\ell$ be a line parallel to $-\mathbf{e}_1$,
let $\CS \subset [0,n_1-1] \times [0, k_1-1]$ be an $\ell$-balanced set for $\eta$.
First suppose $x_1$ may be chosen to lie in $[x_0 + n_1, x_0 + w' - n_1 - 1]$.
Then the restriction of $\eta$ to $[x_0, x_0 + w - 1] \times [y_1 - k_1, y_1 - 1]$
extends nonuniquely to an $\eta$-coloring of 
$\ext_{\ell, \CS}^1([x_0, x_0 + w' - 1] \times [y_1 - k_1, y_1 - 1])$,
and hence also extends nonuniquely to an $(\CS, \eta)$-coloring of that set.
By Lemma~\ref{nonunique-periodic}, it follows that 
$\eta$ is horizontally periodic on $\partial_{\ell, \CS, a}([x_0, x_0 + w' - 1] \times [y_1 - k_1, y_1 - 1])$
with period at most $|E(\ell, \CS)| -1 < n_1$.  Therefore by Lemma~\ref{extend-periodicity}, it is
horizontally periodic with period at most $(2n_1)!$ on 
$$\ext_{\ell', \CS', a}^p(\partial_{\ell, \CS, a}([x_0, x_0 + w' - 1] \times [y_1 - k_1, y_1 - 1])), $$
where $\ell'$ is parallel to $\mathbf{e}_1$ and
$\CS'$ is an $\ell'$-balanced set for $\eta$.
Since $n_1 + a + pa \le a(p+2)$, it follows that $\eta$ is horizontally periodic with period at most
$(n_1)!$ on $[x_0 + a(p+2), x_0 + w - a(p+2)] \times [y_1 - p, y_1 - 1]$.
Thus by the vertical periodicity assumption,
$\eta$ is horizontally periodic on $[x_0 + a(p+2), x_0 + w - a(p+2)]\times [0, h-1]$
with period at most $(2n_1)!$.

Otherwise, $x_1$ cannot be chosen to lie in $[x_0 + n_1, x_0 + w - n_1 - 1]$
but can be chosen to lie in either $[x_0, x_0 + n_1 -1]$ or $[x_0 + w - n_1, x_0 +w - 1]$.
Let us assume it is the former; the argument in the other case is similar.
Let $x_0', y_0'$ be other integers such that $\eta\rst{R + (x_0', y_0')}$ is vertically
periodic with period $p'$.
Assume that~\ref{it:a1} holds for $(x_0',y_0')$ as well and that $y_1'$ is as in~\ref{it:a1}.
Suppose also that $x_1'$ cannot be chosen in $[x_0' + n_1, x_0' + w - n_1 - 1]$
but can be chosen in $[x_0', x_0' + n_1 -1]$.
Assume further that $x_1' - x_0' = x_1 - x_0$.
We claim that $\eta(x_0 + x,y_1 + y) = \eta(x_0'+x,y_1'+y)$ for $(x,y) \in [0,n_i'-1] \times [-p+1,0]$.
Indeed, let $B = [0, n_i - 1] \times [0, k_i - 1]$ and, for 
an integer vector $\mathbf{t} \in [0, n_i - n_i' - 1] \times [1, k_i - p - 1]$,
let $B_{\mathbf{t}} = B + (x_0,y_1 - k_i) + \mathbf{t}$
and
$B_{\mathbf{t}}' = B + (x_0', y_1' - k_i) + \mathbf{t}$.
For a coloring $\alpha\colon B \to \CA$, define
$y(\alpha)$ to be the minimal integer $p' \le y \le k_i-1$ such that
$\alpha(x,y) \neq \alpha(x,y-p')$ for some $0 \le x\le n_i - 1$,
and let $x(\alpha)$ be the maximal such $x$.
Setting $\alpha_{\mathbf{t}} = \eta\rst{B_{\mathbf{t}}}$,
then $(x(\alpha_{\mathbf{t}}), y(\alpha_{\mathbf{t}})) = (x(\alpha_{\mathbf{t}'}), y(\alpha_{\mathbf{t}'}))$
 if and only if $\mathbf{t} = \mathbf{t}'$ and so the colorings $\alpha_{\mathbf{t}}$ are all distinct.
 Similarly, setting $\alpha'_{\mathbf{t}} = \eta\rst{B'_{\mathbf{t}}}$ these colorings of $B$ are also distinct from one another.  
Since there are $(n_i-n_i')(k_i - p)$ choices of $\mathbf{t}$, we have 
$\alpha_{\mathbf{t}} = \alpha'_{\mathbf{t}'}$ for some $\mathbf{t}\neq \mathbf{t}'$. If not, instead we have  
$$2(n_i-n_i')(k_i - p) \ge 4/3n_i (k_i - p) > 4/3 n_i(k_i - k_i/4) = n_i k_i$$
distinct $\eta$-colorings of $B$, a contradiction.
However,  since we assume that $x_1-x_0 = x_1'-x_0'$,
we can have $\alpha_{\mathbf{t}} = \alpha'_{\mathbf{t}'}$
only if $\mathbf{t}= \mathbf{t}'$. Since 
$[x_0,x_0+n_i' - 1] \times [y_1 - p +1 , y_1] \subset B_{\mathbf{t}}$
	for all $\mathbf{t}$, it follows that 
	$\eta(x_0 + x,y_1 + y) = \eta(x_0'+x,y_1'+y)$ for $(x,y) \in [0,n_i'-1] \times [-p+1,0]$,
	as claimed.
	By the vertical periodicity assumptions, there  exists $0 \le j \le p'$ such that
$\eta(x_0 + x, y_0 + y) = \eta(x_0' + x, y_0' + y + j)$ for all $(x,y) \in [0, n_i' - 1] \times [0, h-p-1]$.
Thus, for a pair $(x_0, y_0)$ such that~\ref{it:a1} and~\ref{it:b2} hold,
$\eta\rst{S + (x_0, y_0)}$ is determined by
\begin{itemize}
	\item the vertical period $p' \le p$ of $\eta\rst{R+(x_0, y_0)}$,
	\item the integer $x_1 - x_0 \in [0, n_1 - 1]$, and
	\item the integer $(y_1 - y_0) \mod p' \in [0, p-1]$.
\end{itemize}
 Thus, there are at most $p^2n_1$ possibilities for $\eta\rst{S + (x_0, y_0)}$.
 Arguing similarly, we can bound the number of possibilities if~\ref{it:a2} and~\ref{it:b2} hold,
 and if~\ref{it:b3} holds, all independent of $w$ and $i$. Taking $C$ to be the sum of
 these bounds completes the proof.
\end{proof}

\begin{lemma}
\label{trapezoid}
Suppose $\eta\colon\Z^2\to\CA$ satisfies $P_\eta(n_1,k_1)\leq n_1k_1$ for some $n_1, k_1\in\N$.  
For any $k, m \in \N$ with $k > 14mn_1k_1$, any $\eta$-coloring
of $[0,n_1-1]\times[0,k-1]$ 
either
	\begin{enumerate}
		\item[$(i)$] extends uniquely to 
		an $\eta$-coloring of 
		\begin{multline*}
		T_{m, k} = \\ 
		\Z^2 \cap ([0,n_1-1] \times [0,k-1] 
\cup \{(i,j)\colon n_1 \le i \le m-1,  k_1 (i-n_1) \le j \le k - k_1(i-n_1)\}),
\end{multline*} or
		\item[$(ii)$] extends only to vertically
		periodic (with period independent of $k$ and $m$)
		colorings of
	$$B_{m ,k} = \Z^2 \cap [0,m-1] \times [7mn_1k_1, k - 7mn_1k_1] \subset T_{m,k}.$$
	\end{enumerate}
\end{lemma}

\begin{proof}
Fix an $\eta$-coloring $f$ of $T_{m,k}$.
Let $\ell_1$ and $\ell_2$ be lines parallel to $\mathbf{e}_1$ and $\mathbf{e}_2$ respectively.
By Lemma~\ref{balanced}, there exist balanced sets 
$\CS_1, \CS_2 \subset [0,n_1-1]\times[0,k_1-1]$ for 
$\ell_1$ and $\ell_2$ respectively, with the same width, which is at most $n_1$. 
Let $\tilde{\CS_1} = \CS_1 \setminus E(\ell_1, \CS_1)$
and $\tilde{\CS_2} = \CS_2 \setminus E(\ell_2, \CS_2)$
Suppose $(i)$ does not hold for $f$ restricted to $[0,n_1-1] \times [0, k - 1]$. Letting $d$ denote the smallest integer such that some translate of $\tilde{\CS_1}$
is contained in $[0,d-1]\times \Z$,
there exists $0 \le i \le m - d-1$ such that the restriction
of $f$ to $W_i = T_{m,k} \cap ([i, i + d-1] \times \Z)$
does not extend uniquely to an $\eta$-coloring of $W_i \cup W_{i+1}$.
By Lemma~\ref{nonunique-periodic}, $f\rst{\partial_{\ell_1, \CS_1, 2n_1k_1}(W_i)}$ 
is vertically periodic with period
at most $|E(\ell_1, \CS_1)| - 1 \le k_1$.
Since $\partial_{\ell_1, \CS_1, 2n_1k_1}(W_i)$ contains the rectangle
$R = [i, i+d-1]\times [mk_1+ 3n_1k_1, k - (mk_1 + 3n_1k_1)]$, by Lemma~\ref{extend-periodicity} we have that 
$f$ is vertically periodic, with period at most $[2|E(\ell_1, \CS_1)| - 1)]!$,
on $\ext_{\ell_1,\CS_1, 2n_1k_1}^n(R)$ and 
$\ext_{\ell_2, \CS_2, 2n_1k_1}^n(R)$ for each $n \in \N$.
Since $B_{m,k}$ is a subset of the union of these two sets for $n = m$,
$(ii)$ follows.
\end{proof}

We are now ready to complete the proof of the main theorem: 
\begin{proof}[Proof of Theorem~\ref{trichotomy}]
By Corollary~\ref{exponential-eccentricity} and Remark~\ref{exponential-eccentricity-remark},
without loss of generality we can assume that there exists $C > 1$ such that
$k_i \ge C^{n_i}$ for all $i \in \N$.
In particular, 
$$\dfrac{\log(P_\eta(n_i,k_i))}{k_i} \le \dfrac{2\log(k_i)}{k_i} \to 0.$$
Thus by Proposition~\ref{convexentropy}, $h(\mathbf{e}_2) = 0$.  
Applying Theorem 6.3, Part (4) in~\cite{BL},
if $\mathbf{e}_2$ is an expansive direction then we have that the directional entropy is zero in all directions.  
Thus we can assume that $\mathbf{e}_2$ is nonexpansive, 
and we are left with showing that this is the unique nonexpansive direction and $\eta$ is periodic in this direction. 

For any $m, k \in \N$, the complexity of 
$B_{m,k}$ can written as the sum of the number of colorings
of $[0,n_1-1]\times[0,k-1]$ that extend uniquely to $T_{m, k}$ (and therefore
extend uniquely to the smaller set $B_{m,k}$) plus the number of colorings
of $B_{m,k}$ that do not arise as the unique extension of a coloring of the rectangle $[0,n_1-1]
\times [0,k-1]$.
The number of colorings of the first type is clearly
bounded above by $P_\eta(n_1,k)$. 
By Lemma~\ref{trapezoid}, each of the colorings of $B_{m,k}$ of the latter type is  
vertically periodic with period independent of $m$ and $k$.
Applying Lemma~\ref{few-periodic}, we see that
for sufficiently large $m$ and $k$, the number of such colorings is at most $(C^{1/8})^m$.
Set $k = k_i$ and $m = 8n_i$ for
$i$ large enough such that this bound holds and also
sufficiently large such that $112 n_i n_1 k_1 \le k_i/2$. 
 Then the number of such colorings is at most
$$(C^{1/8})^m  \le n_i C^{n_i} \le n_i k_i.$$
Thus
$$P_\eta(B_{m, k}) \le P_\eta(n_1, k_i) + (C^{1/8})^m \le 
 P_\eta(n_i, k_i) + n_ik_i \le 2n_ik_i.$$
But by the choice of $i$, 
$$|B_{m,k}| = 8n_i (k_i - 14(8n_i)n_1k_1) \ge 8n_i (k_i/2) = 4n_ik_i,$$
and so $P_\eta(B_{m,k}) \le \dfrac{|B_{m,k}|}{2}$.
Hence, by Theorems 1.4 and 1.5 in \cite{CK}, vertical is the unique nonexpansive direction for $\eta$,
and $\eta$ is periodic.  
\end{proof}

\section{Further directions}

We conjecture a stronger result than  Theorem~\ref{trichotomy}, namely that it holds 
under the same assumption as that in Nivat's Conjecture:
\begin{conjecture}
For $\eta\colon\Z^2\to\CA$, if there exist $n,k\in \N$ such that $P_\eta(n,k)\leq nk$, then the directional 
entropy of every nonexpansive direction of $X_\eta$ is zero.  
\end{conjecture}

If the answer is no, this would provide a counterexample to the Nivat Conjecture, and if the 
answer is yes, this is further evidence in favor of the conjecture.  A weaker conjecture would be that 
under the same hypothesis, $X_\eta$ has some direction with zero directional entropy.   Both statements 
follow from Theorem~\ref{trichotomy} under the stronger assumptions on the complexity.  

Alternately it is likely easier to show that a generalization of Nivat's Conjecture, but with a stronger complexity assumption, holds (recall 
Notation~\ref{notation:thick}): 
\begin{conjecture}
 If there exist $K_i \subset \mathbb{R}^2$ compact and convex
with $\lim_{i \to \infty} \frac{\log P_\eta(K_i)}{\tau_\bfu(K_i)} = 0$,
then $\eta$ is periodic.  
\end{conjecture}

Closely related, we ask: 
\begin{question}
\label{isolatedzero}
Say that $\eta$ has an isolated, rational direction of zero directional entropy 
and that $P_\eta(n,k) \le nk$.  Must $\eta$ be periodic?
\end{question}

The following example shows that if the complexity assumption is removed, then the answer is no: 

\begin{example}
Let $\alpha\colon \Z\to \{0,1\}$ such that 
$P_{\alpha}(n) = 2^n$ for all $n \in \N$.
Define $\eta \colon  \Z^2 \to \{0,1,2,3\}$ by $\eta(n,k) = \alpha(n)$ for each $k \neq 0$, and 
$\eta(n,0) = \alpha(n) + 2$.  
Then for the $\Z^2$ action on $X_\eta$ by translations, $h(\bfe_2) = h(-\bfe_2) = 0$,
but $h(\bfu) > 0$ for all other unit vectors $\bfu$.
\end{example}

\begin{proof}
We first prove $h(\bfe_2) = 0$ (the proof that $h(-\bfe_2) = 0$ is analogous).
Fix $t > 0$. There are $2^{2t+1}$ $\alpha$-colorings of $[-t,t]$.
For each of these $\alpha$-colorings $f$,
there are at most $s+2t +2$ $\eta$-colorings of $[-t,t] \times [-t,s+t]$
for which $\eta(i,j) = f(i)$ for some $-t \le j \le s+t$.
Hence, 
$$P_\eta([0,s\bfe_2]^{(t)}) \le P_\eta([-t,t] \times [-t,s+t]) \le 2^{2t+1}(s+2t+2).$$
Thus,
$$\overline{\lim_{s\to \infty}} \frac{\log P_\eta([0,s\bfe_2]^{(t)})}{s}
		 = \overline{\lim_{s\to \infty}} \frac{(2t+1)\log 2 + \log(s+2t+2)}{s} = 0.$$
By Lemma~\ref{entropycomplex}, it follows that $h(\bfe_2) = 0$.

For $\bfu \neq \pm \bfe_2$, let $m = \frac{1}{\|\proj_{\bfe_1} \bfu\|}$ where $\proj_{\mathbf{v}}$ is the projection onto the
direction $\mathbf{v}$.
By assumption $m < \infty$.
Then $\|\proj_{\bfe_1} m\bfu\| = 1$, and 
so $[0,ms\bfu]^{(1)} \cap \{i\} \times \Z \neq \emptyset$
for each $0 \le i \le s$.
For any set $K \subset \mathbb{R}^2$,
if $\vert\{i \in \Z \colon K \cap \{i\} \times \Z\neq \emptyset\} \vert= k_1$
and $\vert\{j \in \Z \colon K \cap \Z \times \{j\}\neq \emptyset\}\vert = k_2$,
then $P_\eta(K) = (k_2+1)2^{k_1}$.
Hence,
\begin{equation*}
\overline{\lim_{s\to \infty}} \frac{\log P_\eta([0,s\bfu]^{(t)})}{s} 
		\ge \overline{\lim_{s\to \infty}} \frac{\log(2^{(s+2t)/m})}{s} = \frac{1}{m}.
		\hfill \qedhere
		\end{equation*}
\end{proof}

Finally, we can ask how much of this holds in higher dimensions.  While there are 
examples~\cite{ST2} showing that the analog Nivat's Conjecture is false for dimension $d\geq 3$, 
it is possible that the results on directional entropy generalize.  
We remark that Cassaigne~\cite{Cas} constructed aperiodic examples  in higher dimensions which 
satisfy the higher dimensional analog of the complexity 
assumptions used in our results.  These all have zero directional entropy in all directions, 
and so do not rule out a higher dimensional version of our theorem: 
\begin{question}
Does the analog of Theorem~\ref{trichotomy} hold for $\eta\colon\Z^d\to\CA$, where $d\geq 3$? 
\end{question}

\end{document}